\documentclass{amsart}

\usepackage{amsmath,amsfonts,amsthm,mathrsfs}
\usepackage{amssymb}

\usepackage[unicode,bookmarks]{hyperref}
\usepackage[usenames,dvipsnames]{xcolor}
\hypersetup{colorlinks=true,citecolor=NavyBlue,linkcolor=BrickRed,urlcolor=Black}

\usepackage[alphabetic,initials]{amsrefs}

\usepackage{enumitem}

\usepackage{chngcntr}

\usepackage{tikz}

\usepackage{tikz-cd}
\tikzset{commutative diagrams/arrow style=math font}

\usetikzlibrary{matrix,arrows}

\newcommand{\tensor}{\otimes}

\newcommand{\NN}{\mathbb{N}}

\newcommand{\QQ}{\mathbb{Q}}
\newcommand{\RR}{\mathbb{R}}

\DeclareMathOperator{\id}{id}

\DeclareMathOperator{\codim}{codim}

\DeclareMathOperator{\Pic}{Pic}

\newcommand{\define}[1]{\emph{#1}}

\newcommand{\gl}{\mathfrak{gl}}

\newcommand{\shf}[1]{\mathscr{#1}}
\newcommand{\OX}{\shf{O}_X}

\def\overbar#1#2#3{{%
	\setbox0=\hbox{\displaystyle{#1}}%
	\dimen0=\wd0
	\advance\dimen0 by -#2 
	\vbox {\nointerlineskip \moveright #3 \vbox{\hrule height 0.3pt width \dimen0}%
		\nointerlineskip \vskip 1.5pt \box0}%
}}

\newcommand{\into}{\hookrightarrow}

\newcommand{\fu}{f^{\ast}}
\newcommand{\fl}{f_{\ast}}

\newcommand{\shF}{\shf{F}}

\newcommand{\shO}{\shf{O}}

\makeatletter
\let\@@seccntformat\@seccntformat
\renewcommand*{\@seccntformat}[1]{%
  \expandafter\ifx\csname @seccntformat@#1\endcsname\relax
    \expandafter\@@seccntformat
  \else
    \expandafter
      \csname @seccntformat@#1\expandafter\endcsname
  \fi
    {#1}%
}
\newcommand*{\@seccntformat@subsection}[1]{%
  \textbf{\csname the#1\endcsname.}
}
\makeatother

\makeatletter
\let\@paragraph\paragraph
\renewcommand*{\paragraph}[1]{%
	\vspace{0.3\baselineskip}%
	\@paragraph{\textit{#1}}%
}
\makeatother

\counterwithin{equation}{subsection}
\counterwithout{subsection}{section}
\counterwithin{figure}{subsection}

\newtheorem{theorem}[equation]{Theorem}
\newtheorem*{theorem*}{Theorem}
\newtheorem{lemma}[equation]{Lemma}
\newtheorem*{lemma*}{Lemma}
\newtheorem{corollary}[equation]{Corollary}
\newtheorem{proposition}[equation]{Proposition}
\newtheorem*{proposition*}{Proposition}
\newtheorem{conjecture}[equation]{Conjecture}

\theoremstyle{definition}
\newtheorem{definition}[equation]{Definition}
\newtheorem*{definition*}{Definition}
\theoremstyle{remark}
\newtheorem*{remark}{Remark}

\newtheorem*{question}{Question}

\newtheorem*{example*}{Example}
\newtheorem*{problem*}{Problem}

\theoremstyle{plain}

\newcommand{\theoremref}[1]{\hyperref[#1]{Theorem~\ref*{#1}}}
\newcommand{\lemmaref}[1]{\hyperref[#1]{Lemma~\ref*{#1}}}
\newcommand{\definitionref}[1]{\hyperref[#1]{Definition~\ref*{#1}}}
\newcommand{\propositionref}[1]{\hyperref[#1]{Proposition~\ref*{#1}}}
\newcommand{\conjectureref}[1]{\hyperref[#1]{Conjecture~\ref*{#1}}}
\newcommand{\corollaryref}[1]{\hyperref[#1]{Corollary~\ref*{#1}}}
\newcommand{\exampleref}[1]{\hyperref[#1]{Example~\ref*{#1}}}

\makeatletter
\let\old@caption\caption
\renewcommand*{\caption}[1]{%
	\setcounter{figure}{\value{equation}}%
	\stepcounter{equation}%
	\old@caption{#1}\relax%
}
\makeatother

\newcounter{intro}

\newtheorem{intro-conjecture}[intro]{Conjecture}
\newtheorem{intro-corollary}[intro]{Corollary}
\newtheorem{intro-theorem}[intro]{Theorem}

\newcommand{\parref}[1]{\hyperref[#1]{\S\ref*{#1}}}

\makeatletter
\newcommand*\if@single[3]{%
  \setbox0\hbox{${\mathaccent"0362{#1}}^H$}%
  \setbox2\hbox{${\mathaccent"0362{\kern0pt#1}}^H$}%
  \ifdim\ht0=\ht2 #3\else #2\fi
  }
\newcommand*\rel@kern[1]{\kern#1\dimexpr\macc@kerna}
\newcommand*\widebar[1]{\@ifnextchar^{{\wide@bar{#1}{0}}}{\wide@bar{#1}{1}}}
\newcommand*\wide@bar[2]{\if@single{#1}{\wide@bar@{#1}{#2}{1}}{\wide@bar@{#1}{#2}{2}}}
\newcommand*\wide@bar@[3]{%
  \begingroup
  \def\mathaccent##1##2{%
    \if#32 \let\macc@nucleus\first@char \fi
    \setbox\z@\hbox{$\macc@style{\macc@nucleus}_{}$}%
    \setbox\tw@\hbox{$\macc@style{\macc@nucleus}{}_{}$}%
    \dimen@\wd\tw@
    \advance\dimen@-\wd\z@
    \divide\dimen@ 3
    \@tempdima\wd\tw@
    \advance\@tempdima-\scriptspace
    \divide\@tempdima 10
    \advance\dimen@-\@tempdima
    \ifdim\dimen@>\z@ \dimen@0pt\fi
    \rel@kern{0.6}\kern-\dimen@
    \if#31
      \overline{\rel@kern{-0.6}\kern\dimen@\macc@nucleus\rel@kern{0.4}\kern\dimen@}%
      \advance\dimen@0.4\dimexpr\macc@kerna
      \let\final@kern#2%
      \ifdim\dimen@<\z@ \let\final@kern1\fi
      \if\final@kern1 \kern-\dimen@\fi
    \else
      \overline{\rel@kern{-0.6}\kern\dimen@#1}%
    \fi
  }%
  \macc@depth\@ne
  \let\math@bgroup\@empty \let\math@egroup\macc@set@skewchar
  \mathsurround\z@ \frozen@everymath{\mathgroup\macc@group\relax}%
  \macc@set@skewchar\relax
  \let\mathaccentV\macc@nested@a
  \if#31
    \macc@nested@a\relax111{#1}%
  \else
    \def\gobble@till@marker##1\endmarker{}%
    \futurelet\first@char\gobble@till@marker#1\endmarker
    \ifcat\noexpand\first@char A\else
      \def\first@char{}%
    \fi
    \macc@nested@a\relax111{\first@char}%
  \fi
  \endgroup
}
\makeatother

\newcommand{\omX}{\omega_X}
\renewcommand{\gl}{g_{\ast}}
\newtheorem{variant}[equation]{Variant}

\begin{document}

\vspace{\baselineskip}

\title{On direct images of pluricanonical bundles}

\author{Mihnea Popa}
\address{Department of Mathematics, Northwestern University,
2033 Sheridan Road, Evanston, IL 60208, USA} 
\email{\tt mpopa@math.northwestern.edu}

\author{Christian Schnell}
\address{Department of Mathematics, Stony Brook University,
Stony Brook, NY 11794, USA}
\email{\tt cschnell@math.sunysb.edu}


\thanks{MP was partially supported by the NSF grant DMS-1101323, and CS by
grant DMS-1331641 and the SFB/TR45 ``Periods, moduli spaces, and arithmetic of
algebraic varieties'' of the DFG.}


\setlength{\parskip}{0.3\baselineskip}

\maketitle

\begin{abstract}
We show that techniques inspired by Koll\'ar and Viehweg's study of weak positivity, combined with vanishing for log-canonical pairs, lead to new generation and vanishing results for direct images of pluricanonical bundles. We formulate the strongest such results as Fujita conjecture-type statements, which are then shown to govern a range of fundamental properties of direct images of pluricanonical and pluriadjoint line bundles, like effective vanishing theorems, weak positivity, or generic vanishing.
\end{abstract}

\tableofcontents

\subsection{Introduction}
The purpose of this paper is twofold: on the one hand we show that techniques inspired by Koll\'ar and Viehweg's study of weak positivity, combined with vanishing theorems for log-canonical pairs, lead to new consequences regarding generation and 
vanishing properties for direct images of pluricanonical bundles. On the other hand, we formulate the strongest 
such results as Fujita conjecture-type statements, which are then shown to govern a range of fundamental properties of direct images of pluricanonical and pluriadjoint line bundles, like effective vanishing theorems, weak positivity, or generic vanishing.

\noindent
{\bf Vanishing, regularity, and Fujita-type statements.}
All varieties we consider in this paper are defined over an algebraically closed
field of characteristic zero. Recall to begin with the following celebrated conjecture.

\begin{conjecture}[Fujita]\label{fujita}
If $X$ is a smooth projective variety of dimension $n$, and $L$ is an ample line bundle on $X$, then 
$\omega_X\otimes L^{\otimes l}$ is globally generated for $l \ge n+1$. 
\end{conjecture}

It is well-known that Fujita's conjecture holds in the case when $L$ is ample and globally generated, based on Kodaira vanishing and the theory of Castelnuovo-Mumford regularity, and that this can be extended to the relative setting as follows:

\begin{proposition}[Koll\'ar]\label{very_ample}
Let $f\colon X \to Y$ be a morphism of projective varieties, with $X$ smooth and $Y$ of dimension $n$. If $L$ is an ample and globally generated line bundle on $Y$, then 
$$R^if_* \omega_X  \otimes L^{\otimes n+1}$$ 
is $0$-regular, and therefore globally generated for all $i$.
\end{proposition}

Recall that a sheaf $\shF$ on $Y$ is \emph{$0$-regular}
with respect to an ample and globally generated line bundle $L$ if
$$H^i (Y, \shF \otimes L^{\otimes -i} ) = 0 \,\,\,\,{\rm ~for~all} \,\,\,\, i > 0.$$ 
The Castelnuovo-Mumford Lemma says that every $0$-regular sheaf is globally generated (see e.g. \cite{Lazarsfeld}*{Theorem 1.8.3}); the Proposition is then a consequence of Koll\'ar's vanishing theorem, 
recalled as Theorem \ref{kollar_vanishing} below.

An extension of Fujita's general conjecture to the relative case was formulated by Kawamata \cite{Kawamata1}*{Conjecture 1.3}, and proved in dimension up to four; the statement is that Proposition \ref{very_ample} should remain true for any $L$ ample, at least as long as the branch locus of the morphism $f$ is a divisor with simple normal crossings support (when the sheaves $R^i f_* \omega_X$ are locally free \cite{Kollar}). However, at least for $i =0$, we propose the following unconditional extension of Conjecture \ref{fujita}:
 
\begin{conjecture}\label{generalized_fujita}
Let $f\colon X\to Y$ be a morphism of smooth projective varieties, with $Y$ of dimension $n$, 
and let $L$ be an ample line bundle on $Y$. Then, for every $k\ge 1$, the sheaf 
$$f_* \omega_X^{\otimes k} \otimes L^{\otimes l}$$
is globally generated for $l \ge k(n+1)$.
\end{conjecture}

Our main result in this direction is a proof of a stronger version of Conjecture \ref{generalized_fujita} in the case of ample and globally generated line bundles, generalizing Proposition \ref{very_ample} for $i = 0$ to arbitrary powers.

\begin{theorem}\label{main}
Let $f\colon X\to Y$ be a morphism of projective varieties, with $X$ smooth and $Y$ of dimension $n$. If $L$ is an ample and globally generated line bundle on $Y$, and $k \ge 1$ an integer, then
 $$f_* \omega_X^{\otimes k}  \otimes L^{\otimes l}$$
is $0$-regular, and therefore globally generated, for $l \ge k(n+1)$.
\end{theorem}

This follows in fact from a more general effective vanishing theorem for direct
images of powers of canonical bundles, which is Koll\'ar vanishing when
$k=1$; see Corollary \ref{vanishing_main}. We also observe in Proposition
\ref{ample_powers} that just knowing the (klt version of the) Fujita-type Conjecture \ref{generalized_fujita} for 
$k =1$ would imply a similar vanishing theorem when $L$ is only ample. 
Using related methods, we find analogous statements in the contexts of
pluriadjoint bundles and of log-canonical pairs as well.
We will call a \emph{fibration} a surjective morphism whose general fiber is irreducible. 

\begin{variant}\label{main_adjoint}
Let $f\colon X\to Y$ be a fibration between projective varieties, with $X$ smooth and $Y$ of dimension $n$. Let $M$ be a nef and $f$-big line bundle on $X$. If $L$ is an ample and globally generated line bundle on $Y$, and $k \ge 1$ an integer, then 
$$f_* (\omega_X\otimes M)^{\otimes k}  \otimes L^{\otimes l}$$ 
is $0$-regular, and therefore globally generated, for $l \ge k(n+1)$.
\end{variant}

\begin{variant}\footnote{The is of
course a generalization of Theorem \ref{main}. We chose to state it separately 
in order to preserve the simplicity of the main point, as will be done a few times throughout the paper.}\label{main_lc}
Let $f\colon X\to Y$ be a morphism of projective varieties, with $X$ normal and 
$Y$ of dimension $n$, and consider a log-canonical 
$\RR$-pair $(X, \Delta)$ on $X$. Consider a line bundle $B$ on $X$ such that 
$B \sim_{\RR} k(K_X + \Delta + f^* H)$ for some $k \ge 1$, where $H$ is an ample $\RR$-Cartier $\RR$-divisor on $Y$.
If $L$ is an ample and globally generated line bundle on $Y$, and $k \ge 1$ an integer, then
 $$f_* B  \otimes L^{\otimes l} \,\,\,\,{\rm with} \,\,\,\, l \ge k (n+1 - t)$$ 
is $0$-regular, and so globally generated, where 
$t : = {\rm sup}~\{s \in \RR ~|~ H - s L {\rm ~is~ample}\}$.
\end{variant}

All of these results are consequences of our main technical result, stated next. It can be seen both as an effective vanishing theorem for direct images of powers, and as a partial extension of Ambro-Fujino vanishing (recalled as Theorem \ref{ambro_fujino} below) to arbitrary log-canonical pairs.

\begin{theorem}\label{vanishing_main_lc}
Let $f\colon X\to Y$ be a morphism of projective varieties, with $X$ normal and 
$Y$ of dimension $n$, and consider a log-canonical 
$\RR$-pair $(X, \Delta)$ on $X$. Consider a line bundle $B$ on $X$ such that 
$B \sim_{\RR} k(K_X + \Delta + f^* H)$ for some $k \ge 1$,
where $H$ is an ample $\RR$-Cartier $\RR$-divisor on $Y$.
If $L$ is an ample and globally generated line bundle on $Y$, then
$$H^i (Y,  f_* B  \otimes L^{\otimes l})= 0 \,\,\,\, {\rm for~all~}i >0 \,\,\,\,{\rm and} \,\,\,\, 
l \ge (k-1)(n+1 - t) - t +1,$$ 
where $t : = {\rm sup}~\{s \in \RR ~|~ H - s L {\rm ~is~ample}\}$.
\end{theorem}

The proof of this result relies on a variation of a method used by Viehweg in the study of weak positivity, and on the 
use of the Ambro-Fujino vanishing theorem.  Shifting emphasis from weak positivity to vanishing turns out to lead to stronger statements, as it was already pointed out by Koll\'ar \cite{Kollar}*{\S3} in the case $k=1$; his point of view, essentially based on regularity, is indeed a crucial ingredient  in the applications. 

One final note in this regard is that all the vanishing theorems used in the paper hold for higher direct images as well. At the moment we do not know however how to obtain statements similar to those above for higher direct images, for instance for $R^i f_* \omega_X^{\otimes k}$ with $i >0$.

\noindent
{\bf Applications.}
The Fujita-type statements in Theorem \ref{main} and its variants turn out to govern a number of fundamental 
properties of direct images of pluricanonical and pluriadjoint bundles. Besides the vanishing statements discussed above, we sample a few here, and refer to the main body of the paper for full statements. To begin with, we deduce in \S4 an effective version of Viehweg's weak positivity theorem for sheaves of the form $f_*\omega_{X/Y}^{\otimes k}$ for arbitrary $k \ge 1$, just as Koll\'ar did in the case $k=1$; we leave the rather technical statement, Theorem \ref{generic}, for the main text.  The same method applies to pluriadjoint bundles, see Theorem \ref{generic_adjoint}, and in this case even the non-effective weak positivity consequence stated below is new. The case $k=1$ is again due to Koll\'ar and Viehweg; see also \cite{Hoering}*{Theorem 3.30} for a nice exposition.

 \begin{theorem}\label{WP_adjoint}
If $f\colon X \to Y$ is a fibration between smooth projective varieties, and $M$ is a nef and $f$-big line bundle on $X$, 
then $f_* (\omega_{X/Y}\otimes M)^{\otimes k}$ is weakly positive for every $k \ge 1$. 
\end{theorem} 

With this result at hand, Viehweg's machinery for studying Iitaka's conjecture can be applied to 
deduce the adjoint bundle analogue of his result on the additivity of the Kodaira dimension over a base of 
general type.

\begin{theorem}\label{Iitaka}
Let $f\colon X\to Y$ be a fibration between smooth projective varieties, and let $M$ be a nef and $f$-big line bundle on $X$. 
We denote by $F$ the general fiber of $f$, and by $M_F$ the restriction of $M$ to $F$. Then:

\noindent
(i) If $L$ is an ample line bundle on $Y$, and $k >0$, then
$$\kappa \left( (\omega_X\otimes M)^{\otimes k} \otimes f^*L \right) = 
\kappa (\omega_F \otimes M_F) + \dim Y.$$

\noindent
(ii) If $Y$ is of general type, then 
$$\kappa (\omega_X \otimes M) = \kappa (\omega_F \otimes M_F) + \dim Y.$$
\end{theorem}
 
In a different direction, the method involved in the proof of Theorem \ref{main}
(more precisely Corollary \ref{vanishing_main}) leads to a generic vanishing statement
for pluricanonical bundles. Let $f\colon X\to A$ be a morphism from a smooth
projective variety to an abelian variety. In \cite{Hacon},  Hacon showed that the
higher direct images $R^i f_* \omega_X$ satisfy generic vanishing, i.e.  are
GV-sheaves on $A$; see Definition \ref{GV_definition}. This refines the well-known
generic vanishing theorem of Green and Lazarsfeld \cite{GL1}, and is crucial in
studying the birational geometry of irregular varieties. In \S5
we deduce the following statement, which is somewhat surprising given our previous
knowledge about the behavior of powers of $\omega_X$. 

\begin{theorem}\label{GV}
If $f\colon X \to A$ is a morphism from a smooth projective variety to an abelian variety, then $f_*\omega_X^{\otimes k}$ is a GV-sheaf for every $k \ge 1$. 
\end{theorem}

We also present a self-contained proof of this theorem based on an effective result, Proposition \ref{summand}, which is weaker than Corollary \ref{vanishing_main}, but has a more elementary proof of independent interest.
Theorem \ref{GV} leads in turn to vanishing and generation consequences that are stronger than those for morphisms to arbitrary varieties; see Corollary \ref{gv_consequences}. Similar statements are given for log-canonical pairs and for adjoint bundles in Variants \ref{GV_lc} and \ref{GV_adjoint}.

\subsection{Vanishing and freeness for direct images of pluri-log-canonical bundles}\label{section_main}
In this section we address results related to Conjecture \ref{generalized_fujita}, via vanishing theorems for direct images 
of pluricanonical bundles. The most general result we prove is for log-canonical pairs; this is of interest from a 
different perspective as well, as it partially extends a vanishing theorem of Ambro and Fujino. 

\noindent
{\bf Motivation and background.}
To motivate the main technical result, recall that given an ample line bundle $L$ on a smooth projective variety of dimension $n$, Conjecture \ref{fujita} implies that $\omega_X \otimes L^{\otimes n+1}$ is a nef line bundle; this is in fact follows unconditionally from the fundamental theorems of the MMP. As a consequence, Kodaira Vanishing implies that for every $k \ge 1$ one has 
\begin{equation}\label{vanishing_powers}
H^i (X, \omega_X ^{\otimes k} \otimes L^{\otimes k(n+1) - n} ) = 0 \,\,\,\,{\rm for~all~} i >0,
\end{equation}
an effective vanishing theorem for powers of $\omega_X$.

We will look for similar results for direct images. Recall first that for $k=1$ there is a well-known analogue of Kodaira vanishing, for all higher direct images.

\begin{theorem}[Koll\'ar Vanishing, \cite{Kollar}*{Theorem 2.1}]\label{kollar_vanishing}
Let $f: X \rightarrow Y$ be a morphism of projective varieties, with $X$ smooth. If $L$ is an ample line bundle 
on $Y$, then
$$H^j (Y, R^i f_* \omega_X\otimes L) = 0 \,\,\,\,{\rm for~all~} i {\rm ~and ~all~} j>0.$$ 
\end{theorem}

Moreover, of great use for the minimal model program are extensions of vanishing and positivity theorems to the log situation, in particular to log-canonical pairs; see e.g. \cite{Fujino3}, \cite{Fujino}. For instance, 
Koll\'ar's theorem above has an extension to this situation due to Ambro and Fujino; see e.g. 
\cite{Ambro}*{Theorem 3.2} and \cite{Fujino3}*{Theorem 6.3} (where a relative version can be found as well).

\begin{theorem}[Ambro-Fujino]\label{ambro_fujino}
Let $f\colon X \to Y$ be a morphism between projective varieties, with $X$ smooth and $Y$ of dimension $n$. 
Let $(X, \Delta)$ be a log-canonical log-smooth $\RR$-pair,\footnote{Meaning that $\Delta$ is an effective $\RR$-divisor with simple normal crossings support, and with the coefficient of each component at most equal to $1$.} and consider a line bundle $B$ on $X$ such that 
$B \sim_{\RR} K_X + \Delta + f^* H$, where $H$ is an ample $\RR$-Cartier $\RR$-divisor on $Y$. Then
$$H^j (Y, R^i f_* B) = 0 \,\,\,\,{\rm for~all~} i {\rm ~and~all~} j>0.$$ 
\end{theorem}

Just as via the Castelnuovo-Mumford Lemma Proposition \ref{very_ample} follows from Theorem \ref{kollar_vanishing}, so 
does Theorem \ref{ambro_fujino} have the following consequence:

\begin{lemma}\label{very_ample_lc}
Under the hypotheses of Theorem \ref{ambro_fujino}, consider in addition an ample and globally generated 
line bundle $L$ on $Y$. Then
$$R^if_* B  \otimes L^{\otimes n}$$ 
is $0$-regular, and therefore globally generated, for all $i$.
\end{lemma}

\noindent
{\bf Main technical result.}
We will now prove our main vanishing theorem, which 
can be seen as an extension of both vanishing of type ($\ref{vanishing_powers}$) for powers of canonical bundles, 
for $L$ ample and globally generated, and of  Ambro-Fujino vanishing Theorem \ref{ambro_fujino} to log-canonical pairs with arbitrary Cartier index.

\noindent
\begin{proof}[Proof of Theorem \ref{vanishing_main_lc}]
{\em Step 1.} We will first show that we can reduce to the case when $X$ is smooth, $\Delta$ has simple normal crossings
support, and the image of the adjunction morphism
\[
	\fu \fl B \to B,
\]
is a line bundle.
A priori the image is $\frak{b} \otimes B$, where $\frak{b}$ is the relative base ideal of $B$. 
We consider a birational modification
$$\mu: \tilde X \longrightarrow X$$
which is a common log-resolution of $\frak{b}$ and $(X, \Delta)$. On $\tilde X$ we can write 
$$K_{\tilde X} - \mu^* (K_X + \Delta) = P - N,$$
where $P$ and $N$ are effective $\RR$-divisors with simple normal crossings support, without common components, and such that $P$ is exceptional and all coefficients in $N$ are at most $1$. We consider the line bundle 
$$\tilde B : =  \mu^* B \otimes \shO_{\tilde X} (k \lceil P \rceil).$$
Note that by definition we have 
$$\tilde B \sim_{\RR} k( K_{\tilde X} + N + \lceil P \rceil - P + \mu^* f^* H).$$
Since $\lceil P \rceil$ is $\mu$-exceptional, we have $\mu_* \tilde B \simeq B$ for all $k$.
Moreover, $\Delta_{\tilde X} : = N + \lceil P \rceil - P$ is log-canonical with simple normal crossings support on $\tilde X$.

Going back to the original notation, we can thus assume that $X$ is smooth and $\Delta$ has simple normal crossings support, and the image sheaf of the adjunction morphism is of the
form $B \tensor \OX(-E)$ for a divisor $E$ such that $E + \Delta$ has simple normal crossings support. 

\noindent
{\em Step 2.}
Now since $L$ is ample, there is a smallest integer $m \ge 0$ such that $f_*B \otimes 
L^{\otimes m}$ is globally generated, and so using the adjunction morphism we have that 
$B \otimes \OX(-E) \otimes \fu L^{\otimes m}$ is globally generated as well. We can then write
\[
	B \tensor \fu L^{\otimes m} \simeq \OX(D+E),
\]
where $D$ is an irreducible smooth divisor, not contained in the support of $E + \Delta$, and such that $D + E + \Delta$ has simple normal crossings support. Rewriting this in divisor notation, we have 
$$k(K_X + \Delta + f^*H)  + mf^* L \sim_{\RR} D + E,$$
and hence
\begin{equation}\label{equivalence}
(k-1) (K_X + \Delta + f^*H) \sim_{\RR}  \frac{k-1}{k} D + \frac{k-1}{k} E - \frac{k-1}{k}\cdot m f^*L.
\end{equation} 

Note that $\Delta$ and $E$ may have common components in their support, which may cause trouble later on; therefore their 
coefficients need to be adjusted conveniently. Let's start by writing 
$$\Delta = \sum_{i=1}^l a_i D_i, \,\,\,\,a_i \in \RR\,\,\,\,{\rm with} \,\,\,\, 0< a_i \le 1$$
and 
$$E = \sum_{i=1}^l s_i D_i + E_1, \,\,\,\,s_i \in \NN,$$
where the support of $E_1$ and that of $\Delta$ have no common components.

Observe now that for every effective Cartier divisor $E^\prime \preceq E$ we have 
\begin{equation}\label{subtraction}
\fl \left( B \tensor \OX(- E^\prime) \right) \simeq \fl B.
\end{equation}
Indeed, it is enough to have this for $E$ itself; but this is the base locus of $B$ relative to $f$, so by construction we have
\[
	\fu \fl B \to B \tensor \OX(-E) \into B,
\]
and so the isomorphism follows by pushing forward to get the commutative diagram
\[
\begin{tikzcd}[column sep=small]
\fl B \arrow[bend left=25]{rr}{\id} \rar 
		& \fl \left( B \tensor \OX(-E) \right) \rar[hook] & \fl B.
\end{tikzcd}
\]
Define now 
$$\gamma_i : = a_i + \frac{k-1}{k} \cdot s_i  \,\,\,\, {\rm for} \,\,\,\, i = 1, \ldots, l.$$ 
We claim that we can find for each $i$ an integer $b_i$ such that 
$$0 \le \gamma_i  - b_i \le 1 \,\,\,\,{\rm and } \,\,\,\, 0 \le b_i \le s_i$$
This is the same as $\gamma_i - 1 \le b_i \le \gamma_i$, while on the other hand, 
$\gamma_i  < 1 + s_i$, so it is clear that such integers exist. We define
$$E^\prime : = \sum_{i =1}^l b_i D_i + \left\lfloor \frac{k-1}{k} E_1 \right\rfloor  \preceq E,$$
and for this divisor ($\ref{subtraction}$) applies.

\noindent
{\em Step 3.}
Using ($\ref{equivalence}$), for any integer $l$ we can now write 
$$B - E^\prime + lf^*L \sim_{\RR} K_X + \Delta +  \frac{k-1}{k} E - E^\prime +  
\frac{k-1}{k}D+ f^* \left(H + \left(l - \frac{k-1}{k}\cdot m \right)L \right).$$
We first note that the $\RR$-divisor 
$$H^\prime : = H + \left( l - \frac{k-1}{k} \cdot m \right) L$$
on $Y$ is ample provided $l + t - \frac{k-1}{k} \cdot m > 0$. On the other hand, the effective $\RR$-divisor with simple normal 
crossings support 
$$\Delta^\prime : = \Delta +  \frac{k-1}{k} E - E^\prime +  \frac{k-1}{k}D$$
on $X$ is log-canonical. Indeed, the only coefficients that could cause trouble are those of the $D_i$. Note however that these are
equal to $\gamma_i - b_i$, which are between $0$ and $1$ by our choice of $b_i$.
Putting everything together, it means that on $X$ which is now smooth we have written 
$$B - E^\prime + lf^*L \sim_{\RR} K_X + \Delta^\prime + f^*H^\prime,$$
where $\Delta^\prime$ is log-canonical with simple normal crossings support, and $H^\prime$ is ample on $Y$.  
The push-forward of the left-hand side is 
$f_* B \otimes L^{\otimes l}$, while for the right-hand side we can now apply Theorem \ref{ambro_fujino} to conclude  that
\begin{equation}\label{vanishing}
H^i (Y,  f_* B  \otimes L^{\otimes l})= 0 \,\,\,\, {\rm for~all~}i >0 \,\,\,\,{\rm and} \,\,\,\, 
 l > \frac{k-1}{k} \cdot m - t .
 \end{equation}
We therefore have that for every $l > \frac{k-1}{k} \cdot m - t + n$ the sheaf $f_* B \otimes L^{\otimes l}$ is $0$-regular, hence 
globally generated. Given our minimal choice of $m$, we conclude that for the smallest integer $l_0$ which is greater than
$\frac{k-1}{k} \cdot m - t $ we have $m \le l_0+ n$. This implies
$$m \le l_0 + n \le  \frac{k-1}{k} \cdot m + n + 1 - t,$$
which is equivalent to $m \le k(n+1 - t)$, and in particular the vanishing in  ($\ref{vanishing}$) holds for 
$$l \ge (k-1) (n+1 - t) - t + 1.$$ 
\end{proof}

Note that the inequality $m \le k(n+1 - t)$ obtained above implies the statement of Variant \ref{main_lc}.
Just as with the statement of Theorem \ref{vanishing_main_lc} compared to that of the Ambro-Fujino Theorem \ref{ambro_fujino}, 
one notes that even for $k = 1$ Variant \ref{main_lc} is slightly more general than the case $i = 0$ in Lemma \ref{very_ample_lc}. This is not surprising, but particular to $0$-th direct images, as after passing to a log-resolution of the log-canonical pair there is no need to appeal to local vanishing for higher direct images; the Ambro-Fujino theorem and the Lemma cannot be stated in this form in the case $i >0$.

\noindent
{\bf Special cases.} 
We spell out the most important special cases of Theorem \ref{vanishing_main_lc}. They are obtained by taking $H = L$ in the statement of Theorem \ref{vanishing_main_lc}, so that $t =1$.

\begin{corollary}\label{lc_special_main}
Let $f\colon X\to Y$ be a morphism of projective varieties, with $X$ normal and 
$Y$ of dimension $n$. Consider a log-canonical pair $(X, \Delta)$ and an integer $k >0$ such that $k (K_X + \Delta)$ is Cartier.
If $L$ is an ample and globally generated line bundle on $Y$, then
 $$H^i (Y,  f_* \shO_X \left(k (K_X + \Delta)\right) \otimes L^{\otimes l})= 0 \,\,\,\, {\rm for~all~}i >0 \,\,\,\,{\rm and} \,\,\,\, 
 l \ge k (n+1) - n.$$ 
\end{corollary}

In particular we have an extension of ($\ref{vanishing_powers}$) to direct images, and of Proposition \ref{very_ample} for $i=0$ to arbitrary $k$:

\begin{corollary}\label{vanishing_main}
Let $f\colon X\to Y$ be a morphism of projective varieties, with $X$ smooth and 
$Y$ of dimension $n$. If $L$ is an ample and globally generated line bundle on $Y$, and $k>0$ is an integer, then
 $$H^i (Y,  f_* \omega_X^{\otimes k} \otimes L^{\otimes l})= 0 \,\,\,\, {\rm for~all~}i >0 \,\,\,\,{\rm and} \,\,\,\, 
 l \ge k (n+1) - n.$$ 
\end{corollary}

\medskip

\begin{remark}\label{only_klt}
Note that if we perform the proof of Theorem \ref{vanishing_main_lc} only in the ``classical" case considered in Corollary \ref{vanishing_main}, the second step is unnecessary since $\Delta = 0$, while $\Delta^\prime$ in 
the third step is klt. This means that one does not need to appeal to the Ambro-Fujino vanishing theorem, but rather to the 
klt version of Theorem \ref{kollar_vanishing}, still due to Koll\'ar; see for instance \cite{shafarevich_maps}*{Theorem 10.19}.
\end{remark}

The regularity statement in the Introduction is an immediate consequence.

\noindent
\begin{proof}[Proof of Theorem \ref{main}]
We note that $k (n +1) = k(n+1) - n + n$, and apply the vanishing statement in Corollary \ref{vanishing_main} by successively subtracting $n$ powers of $L$. 
\end{proof}

A rephrasing of Theorem \ref{main} is a useful uniform global generation statement involving powers of relative canonical bundles.

\begin{corollary}\label{uniform_higher}
Let $f\colon X\to Y$ be a morphism of smooth projective varieties, with $Y$ of dimension $n$. If $L$ is an ample and globally generated line bundle on $Y$, $k\ge 1$ an integer, and $A := \omega_Y \otimes L^{\otimes n+1}$, then 
$$f_* \omega_{X/Y}^{\otimes k}  \otimes A^{\otimes k}$$ 
is globally generated.
\end{corollary}

\begin{question}
The arguments leading to Corollary \ref{vanishing_main}, and more generally Theorem \ref{vanishing_main_lc} and its applications, do not extend to higher direct images. It is natural to ask however whether the statements do hold for all $R^i f_* \omega_X^{\otimes k}$ and  analogues, just as Theorem \ref{kollar_vanishing} and Theorem \ref{ambro_fujino} do. 
\end{question}

\noindent
{\bf Example: the main conjecture over curves.}
We record one case when the main Fujita-type Conjecture \ref{generalized_fujita} can be shown to 
hold, namely that when the base of the morphism has dimension one.  This is not hard to check, but it uses important special facts about vector bundles on curves.

\begin{proposition}\label{curves}
Let $f\colon X\to C$ be a morphism of smooth projective varieties, with $C$ a curve, and let $L$ be an ample line bundle on $C$. Then, for every $k\ge 1$, the 
vector bundle 
$$f_* \omega_X^{\otimes k} \otimes L^{\otimes m}$$
is globally generated for $m \ge 2k$.
\end{proposition}
\begin{proof}
First, note that the sheaf in question is locally free (since $C$ is a curve). We can rewrite it as 
$$f_* \omega_X^{\otimes k} \otimes L^{\otimes m} \simeq f_* \omega_{X/C}^{\otimes k} \otimes \omega_C^{\otimes k}\otimes L^{\otimes m}.$$
Now Kawamata's result \cite{Kawamata2}*{Theorem 1} says that 
$f_* \omega_{X/C}^{\otimes k}$ is a semi-positive vector bundle on $C$, while 
$$\deg \omega_C^{\otimes k}\otimes L^{\otimes m} \ge k (2g - 2) + m \deg L \ge 2g,$$
with $g$ the genus of $C$, as $\deg L > 0$. The statement then follows from the following general result.
\end{proof}

\begin{lemma}
Let $E$ be a semi-positive vector bundle and $L$ a line bundle of degree at least $2g$ on a smooth projective curve 
$C$ of genus $g$. Then $E\otimes L$ is globally generated.
\end{lemma}
\begin{proof}
It is enough to show that for every $p \in C$, one has 
$$H^1 (C, E \otimes L \otimes \shO_C (-p)) = 0,$$
or equivalently, by Serre Duality, that there are no nontrivial homomorphisms
$$E \longrightarrow \omega_C \otimes \shO_C(p) \otimes L^{-1}.$$
But the semi-positivity of $E$ means precisely that it cannot have any quotient line bundle of negative degree.
\end{proof}

\begin{remark}
Note the for curves of genus at least $1$, the argument in Proposition \ref{curves} shows in fact 
that $f_* \omega_X^{\otimes k} \otimes L^{\otimes 2}$ is always globally generated. 
\end{remark}

\noindent
{\bf Relative Fujita conjecture and vanishing for ample line bundles.}
It is worth observing that it suffices to know Conjecture \ref{generalized_fujita} and its variants 
for $k =1$ in order to obtain vanishing theorems for twists by line bundles that are assumed to be just ample, and not necessarily globally generated. For simplicity we spell out only the case of pluricanonical bundles, i.e. the analogue of Corollary \ref{vanishing_main}.

\begin{proposition}\label{ample_powers}
Assume that Conjecture \ref{generalized_fujita} holds for $k=1$.\footnote{Or more precisely its klt version: if $Y$ is smooth of dimension $n$, $L$ is ample on $Y$, and $(X, \Delta)$ is a klt pair such that $B = K_X + \Delta + \alpha f^* L$ is Cartier for some $\alpha \in \RR$, 
then $f_* B \otimes L^{\otimes l}$ is globally generated for any $l + \alpha \ge n+1$.} Then for any morphism
$f\colon X\to Y$ of smooth projective varieties, with  
$Y$ of dimension $n$, any ample line bundle $L$ on $Y$, and any integer $k \ge 2$, one has
 $$H^i (Y,  f_* \omega_X^{\otimes k} \otimes L^{\otimes l})= 0 \,\,\,\, {\rm for~all~}i >0 \,\,\,\,{\rm and} \,\,\,\, 
 l \ge k (n+1) - n.$$ 
\end{proposition}
\begin{proof}
This is a corollary of the proof of Corollary \ref{vanishing_main}. Indeed, the only time we used that $L$ is globally generated and not just ample was to deduce the global generation 
of a sheaf of the form $f_*B \otimes L^{\otimes l}$ from its $0$-regularity with respect to $L$, where $B$ is $\QQ$-linearly equivalent with something of the form $K_X + \Delta + \alpha f^*L$ with $(X, \Delta)$ klt and 
$\alpha \in \QQ$; see also Remark \ref{only_klt}. 
The klt version of Conjecture \ref{generalized_fujita} for $k=1$ would then serve as a replacement.
\end{proof}

A natural version of Conjecture \ref{generalized_fujita} can be stated in the log-canonical case,\footnote{In the absolute case, an Angehrn-Siu type statement has been obtained by Koll\'ar \cite{Kollar2}*{Theorem 5.8} 
in the klt case, and further extended by Fujino \cite{Fujino2}*{Theorem 1.1} to the log-canonical setting.} with the same effect regarding the result of Theorem \ref{vanishing_main_lc}, but this would take us far beyond of what is currently known.

\subsection{Vanishing and freeness for direct images of pluriadjoint bundles}
We now switch our attention to direct images 
of powers of line bundles of the form $\omega_X \otimes M$, where $M$ is a nef and relatively big line bundle. 
Recall first that Proposition \ref{very_ample} has the following analogue:

\begin{proposition}\label{very_ample_twisted}
Let $f\colon X \to Y$ be a fibration between projective varieties, with $X$ smooth and $Y$ of dimension $n$. 
Consider a nef and $f$-big line bundle $M$ on $X$, and $(X, \Delta)$ a klt pair with $\Delta$ and $\RR$-divisor with simple 
normal crossings support. If $B$ is a line bundle on $X$ such that 
$B\sim_{\RR} K_X + M + \Delta + f^*H$ for some ample $\RR$-Cartier $\RR$-divisor $H$ on $Y$, then
$$H^i (Y, f_*B) = 0 \,\,\,\, {\rm for~all~} i>0.$$
In particular, if $L$ is an ample and globally generated line bundle on $Y$, then 
$$f_*B  \otimes L^{\otimes n}$$ 
is $0$-regular, and therefore globally generated.
\end{proposition}
\begin{proof}
We include the well-known proof for completeness, as it is usually given in the case $\Delta = 0$ (see e.g. \cite{Hoering}*{Lemma 3.28}). Note first that $M + f^*H $ continues to be a nef and $f$-big $\RR$-divisor on $X$. The local version of the Kawamata-Viehweg vanishing theorem (see \cite{Lazarsfeld}*{9.1.22 and 9.1.23}) applies then to give
$$R^i f_* B = 0 \,\,\,\, {\rm for~all~} i >0.$$ 
We conclude that it is enough to show 
$$H^i (X, B) = 0 \,\,\,\, {\rm for~all~} i >0.$$
This will follow from the global $\RR$-version of Kawamata-Viehweg vanishing as soon as we show that $M+ f^*H$ is 
in fact a big divisor. Since it is nef, it suffices to check that $(M + f^*H)^m >0$, where $m = \dim X$. Now $(M + f^*H)^m$ is a linear combination with positive coefficients of terms of the form
$$M^s \cdot f^* H^{m -s}$$
which are all non-negative. Moreover, since $M$ is $f$-big, the term $M^{m-n} \cdot f^*H^n$ is strictly positive, which gives the conclusion.
\end{proof}

We now prove an analogue of Corollary \ref{vanishing_main} in this context. Just as with Theorem \ref{main}, Variant \ref{main_adjoint} is its immediate consequence.

\begin{theorem}\label{vanishing_main_adjoint}
Let $f\colon X\to Y$ be a fibration between projective varieties, with $X$ smooth and $Y$ of dimension $n$. Let $M$ be a nef and 
$f$-big line bundle on $X$. If $L$ is an ample and globally generated line bundle on $Y$, and $k \ge 1$ an integer, then
$$H^i (Y,  f_* (\omega_X\otimes M)^{\otimes k} \otimes L^{\otimes l})= 0 \,\,\,\,{\rm for ~all~} i >0 \,\,\,\, {\rm and} 
\,\,\,\, l \ge k(n+1) -n.$$
\end{theorem}
\begin{proof}
\noindent
 The strategy is similar to that of the proof of Theorem \ref{vanishing_main_lc}, so we will be brief in some of the steps.
We consider the minimal $m \ge 0$ such that
$f_* (\omega_X\otimes M)^{\otimes k}  \otimes L^{\otimes m}$
is globally generated. Using the adjunction morphism 
\[
	\fu \fl (\omX\otimes M)^{\otimes k} \to (\omX\otimes M)^{\otimes k}.
\]
after possibly blowing up we can write 
\[
	(\omX \tensor M )^{\otimes k} \otimes f^* L^{\otimes m} \simeq \OX(D+E),
\]
with $D$ smooth and $D+E$ a divisor with simple normal crossings support. In divisor notation, we obtain
\begin{equation}\label{equiv_adjoint}
K_X + M \,\,\,\, \sim_{\QQ} \,\,\,\, \frac{1}{k} D + \frac{1}{k} E - \frac{m}{k} f^*L.
\end{equation}
For any integer $l \ge 0$, using ($\ref{equiv_adjoint}$) we can then write the following equivalence:
$$k(K_X + M) - \left\lfloor \frac{k-1}{k} E \right\rfloor + lf^*L =  K_X + M + (k-1) (K_X + M) - \left\lfloor \frac{k-1}{k} E \right\rfloor + lf^*L$$
$$\sim_{\QQ} \,\,\,\, K_X + M + \Delta + 
\left(l - \frac{k-1}{k}\cdot m \right) f^*L,$$
where
$$\Delta = \frac{k-1}{k} D + \frac{k-1}{k} E - \left\lfloor \frac{k-1}{k} E \right\rfloor $$
is a boundary divisor with simple normal crossings support. Since $E$ is the base divisor of $(\omX\otimes M)^{\otimes k}$ relative to $f$, just as in the proof 
of Theorem \ref{vanishing_main_lc} it follows that 
$$f_* \shO_X \left( k(K_X + M) - \left\lfloor \frac{k-1}{k} E \right\rfloor + lf^*L\right) \simeq  
f_* (\omega_X\otimes M)^{\otimes k}  \otimes L^{\otimes l}.$$
On the other hand, on the right hand side we can apply Proposition \ref{very_ample_twisted}, to deduce that
$$H^i (Y,  f_* (\omega_X\otimes M)^{\otimes k} \otimes L^{\otimes l})= 0 \,\,\,\,{\rm for ~all~} i >0 \,\,\,\, {\rm and} 
\,\,\,\, l > \frac{k-1}{k} \cdot m.$$
We conclude that $f_* (\omega_X\otimes M)^{\otimes k} \otimes L^{\otimes l}$ is globally generated for 
$l > \frac{k-1}{k} \cdot m + n$. Since $m$ was chosen minimal, we conclude as in Theorem \ref{vanishing_main_lc} that
$m\le k(n+1)$, and that vanishing holds for all $l \ge k (n+1) - n$.
\end{proof}

\begin{remark}
Fujita's conjecture and all similar statements have more refined numerical versions, replacing 
$L^{\otimes n+1}$ by any ample line bundle $A$ such that $A^{\dim V} \cdot V > (\dim V)^{\dim V}$ for any subvariety $V \subseteq X$. Similarly, the analogues of Conjecture \ref{generalized_fujita} and Corollary  \ref{ample_powers} make sense replacing $\omega_X$ by $\omega_X \otimes M$ as well. 
\end{remark}

\subsection{Effective weak positivity, and additivity of adjoint Iitaka dimension}\label{section_WP}
We start by recalling the following fundamental definition of Viehweg; see e.g. \cite{Viehweg1}*{\S1}:

\begin{definition}
A torsion-free coherent sheaf $\shF$ on a projective variety $X$ is \emph{weakly positive} on a non-empty open set 
$U \subseteq X$ 
if for every ample line bundle $A$ on $X$ and every $a\in \NN$, the sheaf $S^{[ab]} \shF \otimes A^{\otimes b}$ is generated by 
global sections at each point of $U$ for $b$ sufficiently large. 
(Here $S^{[p]} \shF$ denotes the reflexive hull of the symmetric power $S^p \shF$.)
As noted in \cite{Viehweg1}*{Remark~1.3}, it is not hard to see that 
it is enough to check this definition for a fixed line bundle $A$.
\end{definition}

In \cite{Kollar}*{\S3}, Koll\'ar introduced an approach to proving the weak positivity of sheaves of the form 
$f_* \omega_{X/Y}$ based on his vanishing theorem  for $f_* \omega_X$, which in particular gives effective 
statements. Here we first provide a complement to Koll\'ar's result, using Theorem \ref{main}, in order to make 
this approach work for all $f_* \omega_{X/Y}^{\otimes k}$ with $k \ge 1$.
Concretely, below is the analogue of \cite{Kollar}*{Theorem 3.5(i)}; the proof is very similar, and we only sketch it for convenience.

\begin{theorem}\label{generic}
Let $f\colon X\to Y$ be a surjective morphism of smooth projective varieties, with generically reduced fibers in codimension one.\footnote{This means that there exists a closed subset $Z \subset Y$ of codimension at least $2$ such that over $Y - Z$ the fibers of $f$ are generically reduced. This condition is realized for instance if there is such a $Z$ such that over $Y-Z$ the branch locus of $f$ is smooth, and its preimage is a simple normal crossings divisor; see \cite{Kollar}*{Lemma 3.4}.}
Let $L$ be an ample and globally generated line bundle on $Y$, and $A = \omega_Y
\otimes L^{\otimes n+1}$, where $n = \dim Y$. Then for every $s \geq 1$, the sheaf
$$f_*(\omega_{X/Y}^{\otimes k})^{[\otimes s]}  \otimes A^{\otimes k}$$
is globally generated over a fixed open set $U$ containing the smooth locus of $f$; here  $f_*(\omega_{X/Y}^{\otimes k})^{[\otimes s]}$ denotes the reflexive hull of $f_*(\omega_{X/Y}^{\otimes k})^{\otimes s}$.
\end{theorem}
\begin{proof}
As in \cite{Viehweg1}*{\S3} and in the proof of \cite{Kollar}*{Theorem 3.5}, based on Viehweg's fiber product construction one can show that there is an open set $U \subset Y$, whose complement $Y - U$  has codimension at least $2$,
over which there exists a morphism 
$$\varphi\colon f^{(s)}_* (\omega_{X^{(s)}/Y}^{\otimes k}) \longrightarrow  
f_*(\omega_{X/Y}^{\otimes k})^{[\otimes s]}$$ 
which is an isomorphism over the smooth locus of $f$. Here $\mu \colon X^{(s)} \to X^s$ is a desingularization of the unique irreducible component $X^s$ of the $s$-fold fiber product of $X$ over $Y$ which dominates $Y$; we have natural morphisms 
$f^s: X^s \rightarrow Y$ and $f^{(s)}= f^s \circ \mu: X^{(s)} \rightarrow Y$. The reason one can do this for 
any $k \ge 1$ is this: the hypothesis on the morphism implies that $X^s$ is normal and Gorenstein over such a $U$ (contained in the flat locus of $f$) with complement of small codimension; see also \cite{Hoering}*{Lemma 3.12}. In particular, for every $k \ge 1$ there is a morphism 
$$t: \mu_* \omega_{X^{(s)}/Y}^{\otimes k} \longrightarrow \omega_{X^s/Y}^{\otimes k}$$
which induces $\varphi$.

Now, without changing the notation, we can pass to a compactification of $X^{(s)}$, and the morphism $\varphi$ extends to a morphism of sheaves on $Y$, since it is defined in codimension one and the sheaf on the right is reflexive.
Corollary \ref{uniform_higher} says that $f^{(s)}_* (\omega_{X^{(s)}/Y}^{\otimes k})\otimes A^{\otimes k}$ is globally generated for all $s$ and $k$, which implies that $f_*(\omega_{X/Y}^{\otimes k})^{[\otimes s]} \otimes A^{\otimes k}$  is generated by global sections over the locus where $\varphi$ is an isomorphism.
\end{proof}

\begin{corollary}[Viehweg, \cite{Viehweg1}*{Theorem III}]\label{WP}
If $f\colon X \to Y$ is a surjective morphism of smooth projective varieties, then $f_*\omega_{X/Y}^{\otimes k}$ is weakly positive for every $k \ge 1$. 
\end{corollary} 

This follows in standard fashion from Theorem \ref{generic}, by passing to semistable reduction along the lines of \cite{Viehweg1}*{Lemma 3.2 and Proposition 6.1}. This was already noted by Koll\'ar in the case $k=1$, in \cite{Kollar}*{Corollary 3.7} and the preceding comments. As mentioned above, the Theorem has the advantage of producing an effective bound, at least for sufficiently nice morphisms. We note also that recently Fujino \cite{Fujino4} has used the argument above in order to deduce results on the semipositivity of direct images of pluricanonical bundles.

We now switch our attention to the context of direct images 
of adjoint line bundles of the form $\omega_X \otimes M$, where $M$ is a nef and $f$-big line bundle for a fibration $f\colon X\to Y$.  Given Theorem \ref{vanishing_main_adjoint}, we are now able to use the cohomological approach to weak positivity for higher powers 
of adjoint bundles as well.
Concretely, Theorem \ref{WP_adjoint} again follows via Viehweg's semistable reduction methods 
from the following analogue of the effective Theorem \ref{generic}. 

\begin{theorem}\label{generic_adjoint}
Let $f\colon X\to Y$ be a fibration between smooth projective varieties, with generically reduced fibers in codimension one. 
Let $M$ be a nef and $f$-big line bundle on $X$, $L$ an ample and globally generated line bundle on $Y$, 
and $A = \omega_Y \otimes L^{\otimes n+1}$ with $n = \dim Y$. Then
$$f_*\left( (\omega_{X/Y}\otimes M)^{\otimes k}\right)^{[\otimes s]}  \otimes A^{\otimes k}$$
is globally generated over a fixed nonempty open set $U$ for any $s \ge 1$.
\end{theorem}
\begin{proof}
Using the notation in the proof of Theorem \ref{generic}, over the same open subset $U \subset Y$ with complement of codimension at least $2$, 
one has a morphism which is generically an isomorphism:
\begin{equation}\label{twisted}
\varphi\colon f^{(s)}_* \left((\omega_{X^{(s)}/Y}\otimes M^{(s)})^{\otimes k}\right)\longrightarrow  
f_*\left((\omega_{X/Y}\otimes M)^{\otimes k}\right)^{[\otimes s]}.
\end{equation}
Here $M^{(s)}$ is the line bundle on the desingularization $X^{(s)}$ defined inductively as  
$$M^{(s)} : = p_1^* M \otimes p_2^* M^{(s-1)},$$
with $p_1$ and $p_2$ the projections of $X^{(s)}$ to $X$ and $X^{(s-1)}$ respectively. 
The morphism in ($\ref{twisted}$) is obtained as a consequence of flatness and the projection formula; an excellent detailed discussion of the case $k=1$, as well as of this whole circle of ideas, can be found in \cite{Hoering}*{\S3.D}, in particular Lemma 3.15 and Lemma 3.24. The case $k > 1$ follows completely analogously, given the morphism $t$ in the proof of Theorem 
\ref{generic}.

Finally, Variant \ref{main_adjoint} immediately gives the analogue of Corollary \ref{uniform_higher} for twists by nef and relatively big line bundles, implying that 
$f^{(s)}_* \left((\omega_{X^{(s)}/Y}\otimes M^{(s)})^{\otimes k}\right) \otimes A^{\otimes k}$ is globally generated for all $s$ and $k$. Combined with the reflexivity of the right hand side, this leads to the desired conclusion.
\end{proof}

We conclude by noting that Corollary \ref{WP} has a natural extension to the setting of log-canonical pairs; see 
\cite{Campana}*{\S4}, and also \cite{Fujino}*{\S6}. It is an interesting and delicate problem to obtain an analogue of Theorem \ref{generic} in this setting as well.

\noindent
{\bf Subadditivity of Iitaka dimension for adjoint bundles.}
Theorem \ref{WP_adjoint} allows us to make use of an argument developed by Viehweg in order to provide the analogue in the adjoint setting of his result 
\cite{Viehweg1}*{Corollary IV}  on the subadditivity of Kodaira 
dimension for fibrations with base of general type. 

\noindent
\begin{proof}[Proof of Theorem \ref{Iitaka}]
Note that the $\le $ inequalities are consequences of the Easy Addition formula; see \cite{Mori}*{Corollary 1.7}. The proof of the reverse inequalities closely follows the ideas of Viehweg \cite{Viehweg1} based on the use of weak positivity, as streamlined by Mori with the use of a result of Fujita; we include it below for completeness. Namely, we 
will apply the following lemma (but  not directly for the line bundles on the left hand side in (i) and (ii)).

\begin{lemma}[\cite{Fujita}*{Proposition 1}, \cite{Mori}*{Lemma 1.14}]\label{technical}
Let $f\colon X\to Y$ be a fibration with general fiber $F$, and $N$ a line bundle on $X$. Then there exists a big line bundle $L$ on $Y$ and an integer $m > 0$  with $f^*L \hookrightarrow N^{\otimes m}$ if and only if 
$$\kappa (N)  = \kappa (N_F) + \dim Y.$$ 
\end{lemma}

To make use of this, note first that according to \cite{Viehweg1}*{Lemma 7.3}, there exists a smooth birational modification $\tau\colon Y^\prime \to Y$, and a resolution $X^\prime$ of $X \times_Y Y^\prime$, giving a commutative diagram 
\[
\begin{tikzcd}
X^\prime  \dar{f^\prime} \rar{\tau^\prime} 
		& X  \dar{f} \\
Y^\prime \rar{\tau} & Y		
\end{tikzcd}
\]
with the property that every effective divisor $B$ on $X^\prime$ that is exceptional for $f^\prime$ lies in the exceptional locus of $\tau^\prime$.
Note that in this case $\tau^{\prime}_* \omega_{X^\prime}^{\otimes k} (kB) \simeq \omega_X^{\otimes k}$ for every $k \ge 0$. Also, ${\tau^\prime}^* M$ is still nef and $f^\prime$-big.

Fix now an ample line bundle $L$ on $Y$, and consider the big line bundle $L^\prime = \tau^* L$ on $Y^\prime$. By Theorem \ref{WP_adjoint} we have that for any $k > 0$ (which we can assume to be such that $f^\prime_*(\omega_{X^\prime /Y^\prime }\otimes {\tau^\prime}^* M )^{\otimes k}\neq 0$) there exists $b >0$ such that 
$$S^{[2b]} f^\prime_* (\omega_{X^\prime /Y^\prime }\otimes {\tau^\prime}^* M )^{\otimes k} \otimes {L^\prime}^{\otimes b}$$
is generically globally generated. On the other hand, there exists an effective divisor $B$ on $X^\prime$, exceptional for $f^\prime$, such that the reflexive 
hull of  $f^\prime_* (\omega_{X^\prime /Y^\prime}\otimes {\tau^\prime}^* M )^{\otimes p}$ is equal to $f^\prime_*(\omega_{X^\prime /Y^\prime} (B)\otimes {\tau^\prime}^* M )^{\otimes p}$ for every 
$p \le k b$. Using the nontrivial map induced by multiplication of sections on the fibers, we obtain that 
$$f^\prime_* (\omega_{X^\prime /Y^\prime }(B) \otimes {\tau^\prime}^* M )^{\otimes 2kb}
 \otimes {L^\prime}^{\otimes b}$$
has a non-zero section, and hence we obtain an inclusion
$${f^\prime}^* {L^\prime}^{\otimes b} \hookrightarrow (\omega_{X^\prime /Y^\prime }(B) \otimes {\tau^\prime}^* M )^{\otimes 2kb} \otimes {f^\prime}^* {L^\prime}^{\otimes 2b}.$$
According to Lemma \ref{technical}, we obtain that 
$$\kappa \left( (\omega_{X^\prime /Y^\prime }(B) \otimes {\tau^\prime}^* M )^{\otimes k} 
 \otimes {f^\prime}^* L^\prime \right) = \kappa (\omega_{F^\prime} \otimes  {({\tau^\prime}^* M)}_{F^\prime}) +
 \dim Y^\prime = $$
$$ = \kappa (\omega_F \otimes M_F) + \dim Y,$$ 
 where $F^\prime$ is the general fiber of $f^\prime$. 
 
 To deduce (i) note that, as we have observed that 
 $\tau^{\prime}_* \omega_{X^\prime}^{\otimes k} (kB) \simeq \omega_X^{\otimes k}$, we have 
 $$\tau^\prime_* \left((\omega_{X^\prime /Y^\prime }(B) \otimes {\tau^\prime}^* M )^{\otimes k} 
 \otimes {f^\prime}^* L^\prime \right) \simeq (\omega_{X/Y} \otimes M)^{\otimes k} \otimes f^*L.$$
To deduce (ii), since $Y^\prime$ is of general type recall that by Kodaira's Lemma there exists 
an inclusion $L^\prime \hookrightarrow \omega_{Y^\prime}^{\otimes r}$ for some $r >0$. This implies that 
$$\kappa (\omega_X \otimes M) = \kappa \left( \omega_{X^\prime}(B) \otimes {\tau^\prime}^* M  \right) \ge 
\kappa \left( (\omega_{X^\prime /Y^\prime }(B) \otimes {\tau^\prime}^* M )^{\otimes r} 
 \otimes {f^\prime}^* L^\prime \right),$$
which is equal to $\kappa (\omega_F \otimes M_F) + \dim Y$ by the above.
\end{proof}

\subsection{Generic vanishing for direct images of pluricanonical bundles}\label{section_GV}

We concentrate now on the case of morphisms $f\colon X\to A$, where $X$ is a smooth
projective variety and $A$ is an abelian variety. Recall the following definition
\cite[Definition~3.1]{PP3}.

\begin{definition}\label{GV_definition}
A coherent sheaf $\shF$ on $X$ is called a \define{GV-sheaf} (with respect to the
given morphism $f$) if it satisfies
\[
	\codim \{ \, \alpha \in \Pic^0(A) \, \mid \, H^k(X, \shF \otimes f^* P_{\alpha}) \neq
0 \, \} \geq k
\]
for every $k \geq 0$.
\end{definition}

If $f$ is generically finite, then by a special case of the generic vanishing theorem
of Green and Lazarsfeld \cite{GL1}, $\omega_X$ is a GV-sheaf. This was
generalized by Hacon \cite{Hacon} to the
effect that for an arbitrary $f$ the higher direct images $R^i f_* \omega_X$ are
GV-sheaves on $A$ for all $i$. 
On the other hand, there exist simple examples showing that
even when $f$ is generically finite, the powers $\omega_X^{\otimes k}$ with $k \ge 2$
are not necessarily GV-sheaves; see \cite{PP3}*{Example 5.6}. 
Therefore Theorem \ref{GV} in the introduction is a quite surprising application of 
the methods in this paper. 
 
\begin{proof}[Proof of Theorem \ref{GV}]
Let $M$ be a very high power of an ample line bundle on $\widehat A$, and let 
$\varphi_M\colon \widehat A \to A$ be the isogeny induced by $M$. According to a criterion of 
Hacon \cite{Hacon}*{Corollary~3.1}, the assertion will be proved if we manage to show that 
$$H^i (\widehat A, \varphi_M^*  f_*\omega_X^{\otimes k} \otimes M) = 0 \quad
\text{for all $i >0$.}$$
Equivalently, we need to show that 
$$H^i (\widehat A, g_* \omega_{X_1}^{\otimes k} \otimes M) = 0  \quad
\text{for all $i >0$,}$$
where $g\colon X_1 \to \widehat A$ is the base change of $f\colon X \to A $ via $\varphi_M$.
We can however perform another base change $\mu \colon \widehat A \to \widehat A$ by a multiplication map of large degree, such that $\mu^* M \simeq L^{\otimes d}$, where $L$ is an ample line bundle, which we can also 
assume to be globally generated, and $d$ is arbitrarily large. The situation is summarized in the following 
diagram:
\[
\begin{tikzcd}
X_2  \dar{h} \rar
		& X_1  \dar{g}  \rar & X \dar{f} \\
\widehat A \rar{\mu} & \widehat A \rar{\varphi_M}	& A	
\end{tikzcd}
\]
It is then enough to show that 
$$H^i (\widehat A, h_*\omega_{X_2}^{\otimes k} \otimes L^{\otimes d}) = 0 \quad \text{for all $i
>0$.}$$
Note that we cannot apply Serre Vanishing here, as all of our constructions depend on the original choice of $M$. However, we can conclude if we know that there exists a bound $d = d (n, k)$, i.e. depending only on $n = \dim A$ 
and $k$, such that the vanishing in question holds for any morphism $h$.

At this stage we can of course apply Corollary \ref{vanishing_main}, which allows us to take $d \ge k(n+1) - n$. 
We stress however that as long as we know that such a uniform bound for $d$ exists, for this argument its precise shape does not matter. We therefore choose to present below a weaker but more elementary result that does not need vanishing theorems for $\QQ$-divisors, making the argument self-contained.

Indeed, Proposition \ref{summand} below shows that there exists a morphism  
$\varphi: Z \rightarrow \widehat A$ with $Z$ smooth projective, and $m \le n+k$, such that $h_* \omega_{X_2}^{\otimes k} \otimes L^{\otimes m(k-1)}$ is a direct summand in $\varphi_* \omega_Z$. Applying Koll\'ar vanishing Theorem \ref{kollar_vanishing}, we deduce that 
$$H^i (\widehat A, h_*\omega_{X_2}^{\otimes k} \otimes L^{\otimes d}) = 0 \,\,\,\, {\rm for~ all} \,\,\,\, i
>0 \,\,\,\,{\rm and ~all} \,\,\,\, d \ge (n+k)(k -1) +1,$$
which suffices to conclude the proof.
\end{proof}

\begin{proposition}\label{summand}
Let $f: X \rightarrow Y$ be a morphism of projective varieties, with $X$ smooth and $Y$ of dimension $n$.
Let $L$ be an ample and globally generated line bundle on $Y$, and $k\ge 1$ an integer. Then
there exists a smooth projective variety $Z$ with a morphism $\varphi: Z \rightarrow Y$, and an integer $0 \le m\le n+k$, such that $f_* \omega_X^{\otimes k} \otimes L^{\otimes m(k-1)}$ is a direct summand in $\varphi_* \omega_Z$.
\end{proposition}
\begin{proof}
This is closer in spirit to the arguments towards weak positivity used by Viehweg in \cite{Viehweg1}*{\S5}.
Note first that $f_*\omega_X^{\otimes k} \otimes L^{\otimes pk}$ is globally generated for some sufficiently 
large $p$. Denote by $m$ the minimal $p\ge 0$ for which this is satisfied. 

We are going to use a branched covering construction to
show that $m \leq n +k$. First, consider the adjunction morphism
\[
	\fu \fl \omX^{\otimes k} \to \omX^{\otimes k}.
\]
After blowing up on $X$, if necessary, we can assume that the image sheaf is of the
form $\omX^{\otimes k} \tensor \OX(-E)$ for a divisor $E$ with normal crossing support. As
$\fl \omX^{\otimes k} \tensor L^{\otimes mk}$ is globally generated, we have that the line bundle
\[
	\omX^{\otimes k} \tensor \fu L^{\otimes mk} \tensor \OX(-E)
\]
is globally generated as well. It is therefore isomorphic to $\OX(D)$, where $D$ is an
irreducible smooth divisor, not contained in the support of $E$, such that $D + E$
still has normal crossings. We have arranged that
\[
	(\omX \tensor \fu L^{\otimes m})^{\otimes k} \simeq \OX(D+E),
\] 
and so we can take the associated covering of $X$ branched along $D + E$ and resolve its singularities. This gives us a generically finite morphism $g \colon Z \to X$ of degree $k$, and we denote $\varphi = f\circ g: Z \rightarrow Y$.

Now by a well-known calculation of Esnault and Viehweg  \cite{Viehweg1}*{Lemma 2.3}, the direct
image $\gl \omega_Z$ contains  the sheaf
\begin{align*}
	\omX \tensor (\omX \tensor \fu L^{\otimes m})^{\otimes k-1} \tensor 
		&\OX \left(- \left\lfloor \frac{k-1}{k} \bigl( D+E \bigr) \right\rfloor \right) \\
	\simeq \omX^{\otimes k} \tensor \fu L^{\otimes m(k-1)} 
		\tensor &\OX \left(- \left\lfloor \frac{k-1}{k} E \right\rfloor \right)
\end{align*}
 as a direct summand. If we now apply $\fl$, we find that
\begin{equation} \label{eq:new-sheaf}
	\fl \left( \omX^{\otimes k} \tensor \OX \left(- \left\lfloor \frac{k-1}{k} E \right\rfloor
		\right) \right) \tensor L^{\otimes m(k-1)}
\end{equation}
is a direct summand of $\varphi_{\ast} \omega_Z$. 
At this point we observe as in the proof of Theorem \ref{vanishing_main_lc} that, 
since $E$ is the relative base locus of 
$\omega_X^{\otimes k}$, we have 
\[
	\fl \left( \omX^{\otimes k} \tensor \OX \left(- \left\lfloor \frac{k-1}{k} E \right\rfloor
		\right) \right) \simeq \fl \omX^{\otimes k}.
\]
In other words, $f_* \omega_X^{\otimes k} \otimes L^{\otimes m(k-1)}$ is a direct summand in 
$\varphi_* \omega_Z$. Applying Proposition \ref{very_ample}, we deduce in turn that 
$f_* \omega_X^{\otimes k} \otimes L^{\otimes m(k-1) + n+1}$ is globally generated.
By our minimal choice of $m$, this is only possible if
\[
	m(k-1)+ n + 1 \geq (m-1) k+1,
\]
which is equivalent to $m \leq n + k$. 
\end{proof}

\begin{remark}
With slightly more clever choices, the integer $m$ in Proposition \ref{summand} 
can be chosen to satisfy $m \le n+2$, but the effective vanishing consequence is still weaker than that
obtained in Corollary \ref{vanishing_main}. Note also that one can show analogous results in the case of log-canonical pairs and of adjoint bundles, with only small additional technicalities.  
\end{remark}

Going back to the case when the base is an abelian variety, once we know generic vanishing 
the situation is in fact much better than what we obtained for morphisms to 
arbitrary varieties.

\begin{corollary}\label{gv_consequences}
If $f:X \rightarrow A$ is a morphism from a smooth projective variety to an abelian variety, for every ample line bundle $L$ on $A$ and every $k \ge 1$ one has: 

\noindent
(i) $f_*\omega_X^{\otimes k}$ is a nef sheaf on $A$.

\noindent
(ii) $H^i (A, f_*\omega_X^{\otimes k} \otimes L) = 0$ for all $i >0$.

\noindent
(iii) $f_*\omega_X^{\otimes k} \otimes L^{\otimes 2}$ is globally generated.
\end{corollary}
\begin{proof}
For (i), note that every GV-sheaf is nef by \cite{PP2}*{Theorem 4.1}. Part (ii) follows from the more general 
fact that the tensor product of a GV-sheaf with an IT$_0$ locally free sheaf is IT$_0$; see \cite{PP2}*{Proposition 3.1}. Finally (iii) follows from \cite{PP1}*{Theorem 2.4}, as by (ii)  
$f_*\omega_X^{\otimes k} \otimes L$ is an $M$-regular sheaf on $A$. 
\end{proof}

\begin{question}
It is again natural to ask whether, given a morphism $f\colon X \to A$, the higher direct images $R^i f_* \omega_X^{\otimes k}$ are GV-sheaves for all $i$.
\end{question}

The exact same method, with appropriate technical modifications, gives the following analogues for log-canonical pairs and pluriadjoint bundles, either based on Corollary \ref{lc_special_main} and Theorem \ref{vanishing_main_adjoint}, or on the analogues of Proposition \ref{summand}; we will not repeat the argument.

\begin{variant}\label{GV_lc}
Let $f: X \rightarrow A$ be a morphism from a normal projective variety to an abelian variety. 
If $(X, \Delta)$ is a log-canonical pair and $k \ge 1$ is any integer such that $k(K_X + \Delta)$ is Cartier, 
then $f_* \shO_X \left(k (K_X + \Delta) \right)$ is
a $GV$-sheaf for every $k\ge 1$.
\end{variant}

\begin{variant}\label{GV_adjoint}
Let $f: X\rightarrow A$ be a fibration between a smooth projective variety and an abelian variety, 
and $M$ is a nef and $f$-big line bundle on $X$. Then
$f_*(\omega_X\otimes M)^{\otimes k}$ is a GV-sheaf for every $k \ge 1$. 
\end{variant}

\noindent
{\bf Acknowledgement.} 
CS is very grateful to Daniel Huybrechts for the opportunity to spend the academic
year 2013/14 at the University of Bonn. We also thank the referee for several very
useful corrections.

\end{document}